\documentclass[11pt]{amsart}
\usepackage{amsmath}
\usepackage{amsfonts}
\usepackage{amssymb}

\newtheorem{theorem}{Theorem}[section]

\newtheorem{proposition}[theorem]{Proposition}
\theoremstyle{definition}
\newtheorem{definition}[theorem]{Definition}

\theoremstyle{remark}

\numberwithin{equation}{section}

\begin{document}

\title[Elliptic genus and modular differential equations]
{Elliptic genus and modular differential equations}
\author{Dmitrii Adler}
\address{Euler International Mathematical Institute}
\email{dmitry.v.adler@gmail.com}

\author{Valery Gritsenko}

\address{
Laboratoire Paul Painlev\'e, Universit\'e de Lille and 
International laboratory of mirror symmetry and automorphic forms
NRU HSE}

\email{valery.gritsenko@univ-lille.fr}
\date{\today}

\begin{abstract}
We study modular differential equations for the basic weak Jacobi forms 
in one abelian variable with applications to the elliptic genus of 
Calabi--Yau varieties. 
We show that  the  elliptic genus of any $CY_3$ satisfies a 
differential equation of degree one with respect to the heat operator. 
For a $K3$ surface or any $CY_5$ the degree of the differential 
equation is $3$. We prove that for a general $CY_4$ its elliptic genus   
satisfies a modular differential equation of degree $5$.    
We give examples of  differential equations of degree two with  
respect to the heat operator similar to the Kaneko--Zagier 
equation for modular forms in one variable. We find modular 
differential equations of  Kaneko--Zagier type of degree $2$ or $3$ for 
the second, third and fourth powers of the Jacobi theta-series. \end{abstract}
\maketitle

\section{Elliptic genus and Jacobi modular forms}

\subsection{Elliptic genus of complex varieties with $c_1=0$.}
\label{Sec:2.1}
Jacobi modular forms  appear in geometry and  physics  as special 
partition functions. For example, the elliptic genus of any complex 
compact   variety of dimension $d$ with trivial first Chern class is a 
weak Jacobi form  of weight 
$0$ and index $d/2$ with integral Fourier coefficients (see 
\cite{KYY} in the context of $N = 2$ superconformal field theory, 
\cite{G99a} in the context of Jacobi forms, \cite{T00} in the context 
of elliptic homology). 

Differential equations for modular forms in one variable 
$\tau\in \Bbb H$ were considered in the first half of the 20th century 
by Ramanujan and Rankin (see \cite{Z}).  In this paper we find  
differential equations for the elliptic genera of Calabi--Yau varieties  
of small dimensions. For a Jacobi form of index $d/2$ one has to consider the heat operator   
$
H^{(d)}=4\pi i d\frac{\partial\ }{\partial \tau}-
\frac{\partial^2 \  }{\partial z^2}
$
instead of the differentiation $\frac {d\ }{d\tau}$ 
where $\tau\in \Bbb H$ and $z\in \Bbb C$.
We prove  that  the  elliptic genus of any $CY_3$ satisfies a 
differential equation of degree one with respect to the heat operator.
We find a modular differential equation of degree $3$ in $H^{(1)}$ for 
a $K3$ surface and in $H^{(\frac 52)}$ for any $CY_5$. We prove that 
for a general $CY_4$ there exists a modular differential equation of 
degree $5$ with respect to the heat operator.

We note that modular differential equations are important  
in the description of  the characters of a conformal field theory.
For $N = 2$ superconformal field theories elliptic genera are a
natural generalization of the chiral characters (see \cite{GK}). 
Moreover elliptic genera can be computed geometrically in terms of 
Gromov--Witten invariants in dual string compactifications 
(see  recent papers \cite{OP}, \cite{IOP}).  
The target ring for the elliptic genus of Calabi--Yau varieties is the 
graded ring of the weak Jacobi forms $J_{0,*/2}^{\Bbb Z}$ of weight $0$
and half-integral index with integral Fourier coefficients.
In this paper we find modular differential equations for all generators 
of $J_{0,*/2}^{\Bbb Z}$. 
We find also  modular differential equations for the square,  
the cube, and the fourth power of the Jacobi theta-series $
\vartheta(\tau, z)$.
In \cite{AG}, we  found modular differential equations for  generators of the rings of the weak Jacobi forms for the root lattices $D_2<D_3< \ldots <D_8$. In these terms the subject of this paper is Jacobi forms for the simplest root lattice $A_1$. 
\medskip

Let $M=M_d$ be an (almost) complex compact manifold of (complex) 
dimension $d$ and  $T_X$ be its tangent bundle. 
Let $\tau\in \Bbb H$ be a variable in the upper half-plane and $z\in \Bbb C$.
We set  
$q=\exp(2\pi i \tau)$ and  $\zeta=\exp(2\pi i z)$. We define the formal series
$$
\Bbb {E}_{q,\zeta}=  \bigotimes_{n= 0}^{\infty}
{\bigwedge}_{-\zeta^{-1}q^{n}}T_M^*
\otimes 
 \bigotimes_{n= 1}^{\infty}{\bigwedge}_{-\zeta q^n}  T_M\otimes 
 \bigotimes_{n= 1}^{\infty} S_{q^n} T_M^*
\otimes 
 \bigotimes_{n= 1}^{\infty} S_{q^n} T_M, 
$$
where $\wedge^k$ is the $k^\text{th}$ exterior power, $S^k$ is the $k^\text{th}$ symmetric product, and
$$ 
{\bigwedge}_x E=\sum_{k\ge 0}  (\wedge^k E) x^k, \quad S_x E= \sum_{k\ge 0} (S^k E) x^k.
$$
The elliptic  genus (see \cite{KYY}, \cite{G99b}) 
of $M$ is a function of two variables $\tau\in \Bbb H$ and $z\in \Bbb C$
$$
\chi(M;\tau,z)= \zeta^{d/2}\int_{M}  
\operatorname{ch}(\Bbb E_{q,\zeta})\operatorname{td}(T_{M}),
$$
where  $\operatorname{td}$ is the Todd class, $\operatorname{ch}(\mathbb{E}_{q,\zeta})$
is the Chern  character applied to each  coefficient
of the formal power series, and $\int_M$ denotes the evaluation of 
the top degree differential form on the fundamental cycle of the manifold.
The coefficient $a(n,l)$ of the elliptic genus
$$
\chi(M;\tau,z)=\sum_{n\ge 0,\, l\in \Bbb Z}
a(n,l)\,q^n\zeta^l
$$
is equal to the index of the Dirac operator twisted with the vector bundle 
$E_{n,l-\frac d2}$, where 
$\Bbb {E}_{q,\zeta}=\oplus_{n,l}E_{n,l}q^n\zeta^{l}$.
The $q^0$-term of the elliptic genus of $M$ is equal to the Hirzebruch $\chi_y$-genus (with a small renormalization)
\begin{equation}\label{chiy}
\chi(M_d;\tau,z)=\sum_{p=0}^{d}(-1)^p\chi_p(M)\zeta^{d/2-p}+q(\ldots)
\end{equation}
where $\chi_p(M_d)=\sum_{q=0}^d(-1)^qh^{p,q}(M_d)$.
\begin{theorem}\label{Th:EG}
(See \cite{KYY}, \cite{G99a}.) If $M_d$ is a compact 
complex manifold of dimension $d$ with $c_1(M) = 0$ (over $\Bbb R$), 
then its elliptic genus $\chi(M_d; \tau, z)$ is a weak Jacobi form of 
weight 0 and index $\frac{d}2$ with integral Fourier coefficients.
\end{theorem}

One can define $\chi(M;\tau, z)$  for any compact complex 
variety. It has nice automorphic property only if $c_1(M)=0$.
If $c_1(M)\ne 0$ then one can define an appropriate quasi-modular 
correction of $\chi(M;\tau, z)$  in order to obtain a generalized 
elliptic genus $\chi(M,E;\tau, z)$ for any vector bundle over $M$
(see \cite{G99b} and \cite{Gr20}).
This function is a meromorphic function of Jacobi type. In particular, one gets in this way  the elliptic genera of $(0,2)$-sigma models. 

Theorem \ref {Th:EG} means  that $\chi(M_d; \tau, z)$ satisfies two 
functional equations with respect to the modular group $SL_2(\Bbb Z)$ and the Heisenberg group $H(\Bbb Z)$. We note that 
the definition of $\chi(M_d;\tau,z)$ is a geometric realization of  
the Jacobi triple product formula for the anti-invariant Jacobi theta-series of characteristic two 
\begin{equation}\label{theta}
\begin{split}
\vartheta(\tau ,z)&=
q^{1/8}(\zeta^{1/2}-\zeta^{-1/2})\prod_{n\ge 1}\,(1-q^{n}\zeta)(1-q^n 
\zeta^{-1})(1-q^n)\\
{}&=q^{\frac{1}{8}}\zeta^{\frac{1}{2}}
\sum_{n\in \Bbb Z}\,(-1)^nq^{\frac{n(n+1)}{2}}\zeta^{n}.
\end{split}
\end{equation}
We note that $\vartheta(\tau, -z)=-\vartheta(\tau, z)$ and
$\vartheta(\tau, z)=-i\vartheta_{11}(z,\tau)$ in the notation of 
\cite[Chapter 1]{M}.
The theta-series  $\vartheta(\tau ,z)$ is a Jacobi modular form of 
index $\frac 12$ and it is one of the main instruments of our study.

\subsection{Definition of Jacobi forms}\label{Sec:1.2}
The  natural model of the Jacobi modular group $\Gamma^J(\Bbb Z)$ 
is the quotient  $\Gamma_{\infty}/\{\pm I\}$ of the integral maximal 
parabolic subgroup of the Siegel modular group of genus $2$ fixing an isotropic line. 
We refer to \cite[\S 1]{GN98} and \cite{CG} for this point of view and  for the definition below. The Jacobi group $\Gamma^J(\Bbb Z)$ is the semidirect product of $SL_2(\Bbb Z)$ with the integral Heisenberg group $H(\Bbb Z)$ which is the central extension of $\Bbb Z\times\Bbb Z$. ($H(\Bbb Z)$ is the unipotent subgroup of $\Gamma_\infty(\Bbb Z)$.) 
One can define a binary character of $H(\Bbb Z)$ 
$$
\upsilon_{H}([x,y;r])=(-1)^{x+y+xy+r}.
$$
\begin{definition}\label{def:JF}
Let $k\in \Bbb Z$ and $m\in \frac{1}{2}\Bbb N$. 
A holomorphic function
\newline 
$\varphi : \Bbb H\times \Bbb C\to \Bbb C$ is called a weak Jacobi 
form of weight $k$ and index $m$  if it satisfies  
\begin{align*}
\varphi\left(\frac{a\tau+b}{c\tau+d},\frac{z}{c\tau+d}\right)&
=(c\tau+d)^ke^{2\pi i m \frac{cz^2}{c\tau+d}}\,\varphi(\tau,z),\\
\varphi(\tau,z+x\tau+y)&=v_H([x,y;0])^{2m}e^{-2\pi i m (x^2\tau+2xz)}\varphi(\tau,z),
\end{align*}
for any $A=\left(\begin{smallmatrix}
a & b \\ 
c & d \end{smallmatrix}\right)\in SL_2(\Bbb Z)$ and
$x,y\in \Bbb Z$,
and has a Fourier expansion of the form
\begin{equation}\label{F-exp} 
\varphi(\tau,z)=\sum_{\substack{ n\geq 0,\ l\in \frac{1}{2}\Bbb Z}} a(n,l)\exp(2\pi i(n\tau+lz))
\end{equation}
where $l$ is half-integral for half-integral index $m$. 
If $\varphi$ satisfies the additional condition $a(n,l)=0$ for $4nm - l^2<0$ then it is called a {\it holomorphic} (at infinity) Jacobi form. Moreover, if $\varphi$ satisfies the stronger condition
$f(n,l)=0$ for $4nt - l^2\le0$ then it is called a Jacobi {\it cusp} form.
\end{definition}

In this paper we denote by $J_{k,m}$ the vector space of weak Jacobi 
forms of weight $k$ and index $m$. It is easy to see that 
$J_{k,0}=M_k(SL_2(\Bbb Z))$ is the space of classical modular forms 
with respect to the modular group.

We note that the Jacobi theta-series $\vartheta(\tau ,z)$ introduced 
above is a holomorphic Jacobi form of weight $\frac{1}2$ and index 
$\frac{1}2$ in $J_{\frac{1}2, \frac{1}2}(v_\eta^3\times v_H)$ where 
$v_\eta$ is the multiplier system of the Dedekind eta-function 
\begin{equation}\label{eta}
\eta(\tau)=q^{\frac 1{24}}\prod_{n\ge 1}(1-q^n)
\in M_{\frac{1}2}(SL_2(\Bbb Z),v_\eta).
\end{equation} 

According to \cite[Theorem 9.3]{EZ} the bigraded ring of the weak 
Jacobi forms of even weight and integral index is a polynomial ring 
with four generators 
\begin{equation}\label{J2**}
J_{2*,*}=\bigoplus_{k\in \Bbb Z,\, m\in \Bbb Z_\ge 0}  
J_{2k,m}=
\Bbb C[E_4, E_6, \varphi_{-2,1},\varphi_{0,1}],
\end{equation}
where 
$E_4(\tau)=1+240\sum_{n\ge 1}\sigma_3(n)q^n$ and 
$E_6(\tau)=1-504\sum_{n\ge 1}\sigma_5(n)q^n$
are the classical Eisenstein series of weight $4$ and $6$, and
\begin{equation}\label{phi-21}
\varphi_{-2,1}=\frac{\vartheta(\tau,z)^2}{\eta(\tau)^6}
=\zeta^{\pm 1}-2+ 
q(-2\zeta^{\pm 2} +8\zeta^{\pm 1}-12)+O(q^2),
\end{equation}
\begin{equation}\label{phi01}
\varphi_{0,1}=-\frac{3}{\pi^2}\wp(\tau,z)\varphi_{-2,1}=
\zeta^{\pm 1}+10+ q(10\zeta^{\pm 2} - 64\zeta^{\pm 1}+108)+O(q^2),
\end{equation}
where $\wp(\tau,z)$ is the Weierstrass $\wp$-function.  We note that 
\begin{equation*}\label{G2}
\frac{\partial\vartheta(\tau ,z)}{\partial z}\big|_{z=0}
=2\pi i \,\eta(\tau )^3,\quad
\dfrac{\partial^2 \log \vartheta(\tau,z)}{\partial z^2}=
-\wp(\tau,z)+8\pi^2 G_2(\tau)
\end{equation*}
where $G_2(\tau)=-\frac{1}{24}+\sum_{n\ge 1 }\sigma_1(n)q^n$
is the quasi-modular Eisenstein series.

For Jacobi forms of half-integral weights and indices and for the 
interpretation of $\vartheta(\tau,z)$ as a Jacobi modular form see 
\cite{GN98} and \cite{G99a}. (See also \cite{G94} and \cite{CG} for general case of  Jacobi forms for orthogonal groups.)
We have 
\begin{equation}\label{J032}
J_{-1, \frac{1}2}=\Bbb C\varphi_{-1,\frac{1}2},\quad 
J_{-1,2}=\Bbb C\varphi_{-1,2},\quad 
J_{0,\frac 32}=\Bbb C\varphi_{0,\frac 32}
\end{equation}
where 
\begin{equation}\label{phi32}
 \varphi_{-1,\frac{1}{2}}(\tau,z)=\frac{\vartheta(\tau,z)}
{\eta^3(\tau)}, \ \varphi_{-1,2}(\tau,z)=
\frac{\vartheta(\tau,2z)}{\eta^3(\tau)},\ 
\varphi_{0,\frac{3}{2}}(\tau,z)=\frac{\vartheta(\tau,2z)}
{\vartheta(\tau,z)}.
\end{equation}
Moreover the following relations are true for any $k\in \Bbb Z$ and 
$m\in \Bbb N$ 
\begin{equation}\label{eq:structure1}
J_{2k+1,m}=\phi_{-1,2}\cdot J_{2k+2,m-2},\quad
J_{2k+1,m+\frac{1}{2}}=\varphi_{-1,\frac{1}{2}}\cdot J_{2k+2,m},
\end{equation}
\begin{equation}\label{eq:structure2}
J_{2k,m+\frac{1}{2}}=\varphi_{0,\frac{3}{2}}\cdot J_{2k,m-1}.
\end{equation}
Since the functions $\varphi_{-1,2}$, $\varphi_{0,\frac{3}{2}}$ and 
$\varphi_{-1,\frac{1}{2}}$ have  infinite product expansions, the 
relations 
(\ref{eq:structure1}) and 
(\ref{eq:structure2}) hold for the corresponding spaces of Jacobi forms 
with integral Fourier coefficients.

According to Theorem \ref{Th:EG}
the target ring for the elliptic genus is the graded ring 
$J_{0,*/2}^{\Bbb Z}=\oplus_{m\in \Bbb N}\, J_{0,\frac{m}2}$.
To describe its structure we have to introduce three additional Jacobi 
forms of weight $0$ and index $2$, $3$ and $4$ with integral Fourier 
coefficients
\begin{equation}\label{phi2}
\begin{aligned}
\varphi_{0,2}(\tau,z)
&=\tfrac{\varphi_{0,1}^2-E_4\varphi_{-2,1}^2}{24}=
\zeta^{\pm 1}+4+q(\zeta^{\pm 3}-8\zeta^{\pm 2}-\zeta^{\pm 1}+16)+\ldots,\\
\varphi_{0,3}(\tau,z)&=\varphi_{0,\frac 32}(\tau,z)^2=
\zeta^{\pm 1}+2+2q(\zeta^{\pm 3}-\zeta^{\pm 2}+\zeta^{\pm 1}+2)+
\ldots,\\
\varphi_{0,4}(\tau,z)&=\frac{\vartheta(\tau, 3z)}{\vartheta(\tau, z)}=
\zeta^{\pm 1}+1
-q(\zeta^{\pm 4}+\zeta^{\pm 3}-\zeta^{\pm 1}-2)+\ldots .
\end{aligned}
\end{equation}

\begin{theorem}\label{thm:J0*}(See \cite{G99a}.)
1. $J^{\Bbb Z}_{0,*/2}=J^{\Bbb Z}_{0,*}[\varphi_{0, \frac 32}]$.

\noindent
2. The graded ring $J^{\Bbb Z}_{0,*}$ of the weak Jacobi forms of 
weight $0$ and integral index with integral Fourier coefficients is 
finitely generated over $\Bbb Z$. More precisely,
$$ 
J^{\Bbb Z}_{0,*}
= \Bbb Z\, [\varphi_{0,1},\varphi_{0,2},\varphi_{0,3},\varphi_{0,4}].
$$
The functions $\varphi_{0,1}$, $\varphi_{0,2}$, $\varphi_{0,3}$ are algebraically independent over $\Bbb C$ and 
$4\varphi_{0,4}=\varphi_{0,1}\varphi_{0,3}-\varphi_{0,2}^2$. 
\end{theorem}

\section{Modular differential equations for the elliptic genus of Calabi--Yau varieties of dimension $2$, $3$ and $5$}

Modular differential operators appear naturally in the geometric and physics context of elliptic genera (see \cite{GK}, \cite{GT}, \cite{OP}).
\subsection{The modular differential operator $D_k$.}
Let
$$
D=q\frac{d\ }{dq}=\frac{1}{2\pi i}\frac{d\ }{d\tau}.
$$
For any 
$f(\tau)=\sum_{n\ge n_0\in \Bbb Z} a(n)q^n$ we have
$D(f)=\sum_{n\ge n_0\in \Bbb Z} na(n)q^n$.
Moreover, for any  meromorphic automorphic 
form $f(\tau)$ of weight $0$ we have
$$
D(f)\in M_2^{(mer)}(SL_2(\Bbb Z)).
$$
It gives the following modular differential operator
$$
D_k: M_k(SL_2(\Bbb Z))\to M_{k+2}(SL_2(\Bbb Z))
$$
\begin{equation}
D_k(f)=\eta^{2k}D(\frac f{\eta^{2k}})=D(f)-2k\frac{D(\eta)}{\eta}f=
D(f)-\frac{k}{12}E_2\cdot f
\end{equation}
where
\begin{equation}
\frac{D(\eta)}{\eta}=-G_2(\tau)=\frac{1}{24}-\sum_{n\ge 1}\sigma_1(n)q^n
=\frac{1}{24}E_2(\tau)
\end{equation}
is the quasi-modular Eisenstein series of weight $2$.  

Directly from the definition of $D_k$ we obtain
\begin{equation}\label{D-delta}
D_{\frac k 2}(\eta^k)=0,\quad
D_{4}(E_4)=-\frac{1}3E_6, \quad D_{6}(E_6)=-\frac{1}2E_4^2.  
\end{equation}
In particular, $D_{12}(\Delta(\tau))=0$ where 
$\Delta(\tau)=\eta^{24}(\tau)$ is the Ramanujan cusp form of weight 
$12$.

\subsection{The heat  operator and its modular correction.}

The analogue of the derivation $\frac{1}{2\pi i}\frac{d \ }{d\tau}$
in the case of Jacobi modular forms of index $m$ is the heat 
operator
\begin{equation}\label{heat}
{H}^{(m)}=\frac{3}{m}\frac 1{(2\pi i)^2}
\left( 8\pi i m\dfrac{\partial\ }{\partial \tau}-
\dfrac{\partial^2 \  }{\partial z^2}\right)
=12q\frac {d\ }{dq}- \frac{3}m\left(\zeta\frac {d\ }{d\zeta}\right)^2.
\end{equation} 
We have 
$$
{H}^{(m)}(q^n\zeta^l)=\frac 3{m}(4nm-l^2)q^n\zeta^l,
$$ 
where 
$4nm-l^2$ is the hyperbolic norm of the index $(n,l)$ of  the Fourier 
coefficient $a(n,l)$ (see (\ref{F-exp})).
This normalization of the heat operator with operator $D$ multiplied by 
12 is useful in the context of the differential 
equations. We firstly used it in \cite{AG}. 
To simplify the notation we write $H={H}^{(m)}$ in all cases when the 
exact index of Jacobi modular forms is clear.

Any holomorphic Jacobi forms of weight $\frac{1}2$
(this is the so-called {\it singular weight} for Jacobi modular forms 
of Eichler-Zagier type) lies in the kernel of the heat 
operator. In particular, $H^{(\frac 12)}(\vartheta(\tau,z))=0$. See 
\cite{G94} where the general case of Jacobi forms of orthogonal type 
was considered.
We know that the heat operator $H^{(m)}$ transforms a 
(meromorphic) Jacobi form of the singular weight 
(i.e. weight $\frac 12$ in our case) and index $m$ 
into a Jacobi form of the weight $\frac 12+2$ and the same index 
(see \cite[Lecture 11]{GrJ} for more details).
Similarly to the operator $D_k$ above one can define  the modular 
differential operator for Jacobi forms
\begin{equation}
H_{k}:J_{k,m}\to J_{k+2,m}, \quad
H_{k}(\varphi_{k,m})=H(\varphi_{k,m})
-\frac{(2k-1)}{2}E_2\cdot\varphi_{k,m}
\end{equation}
with the quasi-modular Eisenstein series 
$E_2(\tau)=1-24\sum_{n\ge 1}\sigma_1(n)q^n$. 
This differential operator transforms weak (holomorphic or cusp) Jacobi 
forms into Jacobi forms of the same type. 

\subsection{The elliptic genus of  varieties of dimension $2$, $3$ and 
$5$.}
There are three cases when the space $J_{0,m}$ is one-dimensional. They 
are  $m=1$, $\frac 32$ and $\frac 52$. 
According to Theorem \ref{Th:EG} and (\ref{chiy}) the elliptic 
genus of a Calabi--Yau variety of dimension $2$, $3$ or $5$ depends 
only on its Euler characteristic.

A Calabi--Yau variety of dimension $2$ is a $K3$ surface. Since 
$e(K3)=24$, we have
$$
\chi(K3; \tau, z)
=2\varphi_{0,1}(\tau, z)=2\zeta+20+2\zeta^{-1}+q(\ldots)
$$
where $\varphi_{0,1}\in J_{0,1}$ is one of the two generators of the 
bigraded ring of weak Jacobi forms defined in (\ref{phi01}). 
We note that the Jacobi form $\varphi_{0,1}$ is the elliptic genus of 
an Enriques  surface. For Calabi--Yau varieties of dimension $3$ and 
$5$
we have
\begin{equation}\label{EG:CY5}
\begin{aligned}
\chi(CY_3; \tau, z)
&=\frac {e(CY_3)}2\,\varphi_{0,\frac32}
=\frac {e(CY_3)}2\left(\zeta^{-\frac{1}2}+\zeta^{\frac{1}2}+q(\ldots)\right),\\
\chi(CY_5; \tau, z)&
=\frac{e(CY_5)}{24}\,\varphi_{0,\frac32}\cdot \varphi_{0,1}=
\frac {e(CY_5)}{24}\left(
\zeta^{\pm \frac 32}+11\zeta^{\pm \frac 12}+q(\ldots)\right)
\end{aligned}
\end{equation}
where $e(M_d)$ is the Euler number of $M_d$. The last identity shows 
that the Euler number of a Calabi--Yau variety of dimension $5$ is 
divisible by $24$ (see \cite{G99a}).

\begin{theorem}\label{deq:CY3}
Let $M_3$ be a Calabi--Yau variety of dimension $3$ with $e(M_3)\ne 0$.
Its elliptic genus $\chi(M_d; \tau, z)$ satisfies a modular 
differential equation of degree $1$ with respect to the heat operator 
$H$. More exactly
\begin{equation}\label{deq:CY3}
H(\varphi_{0,\frac 32})+\frac 12 E_2\cdot\varphi_{0,\frac 32}=0.
\end{equation} 
\end{theorem}
\begin{proof} 
We note that $J_{2,\frac{3}2}=M_2(SL_2(\Bbb Z))\phi_{0, \frac{3}2}=\{0\}$. Therefore (\ref{deq:CY3}) follows directly from the equality
$H_{0}(\phi_{0,\frac{3}2})=0$.
\end{proof}

\noindent
{\bf Remark.} 
We can give another  explanation  of (\ref{deq:CY3}). 
According to the quintiple product formula 
(see \cite[Lemma 1.6]{GN98}), 
we have that
$$
\eta(\tau)\varphi_{0,\frac 32}(\tau,z)=
\eta(\tau)\frac{\vartheta(2z)}{\vartheta(z)}=
\sum_{n\in \Bbb Z}\left(\frac n{12}\right)q^{n^2/24}\zeta^{n/2}
\in J_{\frac{1}2, \frac 32}^{hol}(v_\eta\cdot v_H),
$$
where
$\left(\frac n{12}\right)=\pm 1$ or $0$ is the quadratic Legendre 
symbol, is a {\it holomorphic} Jacobi form of singular weight 
$\frac 12$. Therefore 
$$
H(\eta(\tau)\varphi_{0,\frac 32}(\tau,z))=0.
$$
Using the relation $\frac {12D(\eta)}{\eta}=\frac 1{2}E_2(\tau)$
we rewrite this identity as (\ref{deq:CY3}).
\smallskip

We formalize computation used in  the second proof of 
(\ref{deq:CY3}) above. 

\begin{proposition}\label{prop:DH}
Let $f_{k_1}(\tau)\in M_{k_1}(SL_2(\Bbb Z))$ and
$\varphi_{k_2,m}(\tau,z)\in J_{k_2,m}$ where $m\in \frac 12\Bbb Z$. 
We have
$$
H_{k_1+k_2}(f_{k_1}\varphi_{k_2,m})=
12D_{k_1}(f_{k_1})\varphi_{k_2,m}
+f_{k_1}H_{k_2}(\varphi_{k_2,m}).
$$
In particular,
$H_{\frac {n}2+k}(\eta^{n}\varphi_{k,m})=\eta^{n}H_{k}(\varphi_{k,m})$.
\end{proposition}

The cases of Calabi--Yau varieties of dimensions two and five are more 
difficult, but quite similar.

\begin{theorem}\label{deq:dim1}
Let $M_d$ be a Calabi--Yau variety of dimension $2$ or $5$ and 
$e(M_5)\ne 0$. Then the elliptic genus $\chi(M_d; \tau, z)$ satisfies a 
modular differential equation of degree $3$ with respect to the heat 
operator $H$. More exactly we have
\begin{equation}\label{deq:K3}
H_4H_2H_0(\varphi_{0,1})-\frac{101}4E_4H_0(\varphi_{0,1})+10E_6\varphi_{0,1}=0,
\end{equation}
\begin{equation}\label{deq:CY5}
H_4H_2H_0(\varphi_{0,\frac 52})-\frac{611}{25}E_4H_0(\varphi_{0,\frac 52})+\frac{88}{25}E_6\varphi_{0,\frac 52}=0.
\end{equation}
We can rewrite the last two equations using only the heat operator and the Eisenstein series  $E_2$, $E_4$ and $E_6$:
\begin{multline*}
H^3(\varphi_{0,1})-\frac{9}{2}E_2H^2(\varphi_{0,1})
+\left(\frac{9}{4}E_2^2-\frac{99}{4}E_4\right)H(\varphi_{0,1})\\
+\left(\frac{3}{8}E_2^3-\frac{99}{8}E_2E_4+12E_6\right)\varphi_{0,1}=0
\end{multline*}
and
\begin{multline*}
H^3(\varphi_{0,\frac{5}{2}})-\frac{9}{2}E_2H^2(\varphi_{0,\frac{5}{2}})+\left(\frac{9}{4}E_2^2-\frac{1997}{50}E_4\right)H(\varphi_{0,\frac{5}{2}})\\
+\left(\frac{3}{8}E_2^3-\frac{1997}{100}E_2E_4+\frac{138}{25}E_6\right)\varphi_{0,\frac{5}{2}}=0.
\end{multline*}
By $H^3(\phi)$ we denote as usual the operator $H(H(H(\phi))$.
We note that we use the heat operator $H^{(1)}$  for  
$\varphi_{0,1}$ and the heat operator $H^{(\frac 52)}$ for 
$\varphi_{5/2}$ (see (\ref{heat})).
\end{theorem}
\begin{proof} 
Let $\varphi_{k,m}\in J_{k,m}$. By  $J_{k,m}(q)$ we denote the subspace 
of the  weak Jacobi forms in $J_{k,m}$ whose $q^0$-Fourier coefficient 
is equal to $0$.

According to (\ref{heat}), we have a 
simple formula for 
the $q^0$-Fourier coefficient of $H_k(\varphi_{k,m})$. If 
$q^0(\varphi_{k,m})=\sum_{l}\,a(0,l)\zeta^l$ then 
\begin{equation}\label{q0-coef}
q^0(H_k(\varphi_{k,m}))= 
\sum_{l}\,\left(-\frac{3l^2}m-\frac{2k-1}2\right)a(0,l)\zeta^l.
\end{equation}
We know that  $J_{0,1}=\Bbb C\varphi_{0,1}$ and  
$J_{2,1}=\Bbb C E_4\varphi_{-2,1}$. Since 
$q^0(H_2(\varphi_{-2,1}))=-\frac 12 \zeta^{\pm 1}+5$ and 
$q^0(H_0(\varphi_{0,1}))=-\frac 52\zeta^{\pm 1}+5$, we have
\begin{equation}\label{Hphi}
H_{-2}(\varphi_{-2,1})=-\frac 12\varphi_{0,1},\quad
H_{0}(\varphi_{0,1})=-\frac 52 E_4\varphi_{-2,1}.
\end{equation}
Using (\ref{D-delta}), (\ref{Hphi}) and Proposition \ref{prop:DH} we 
obtain
\begin{equation}\label{deq-2:phi1}
H_2H_0(\varphi_{0,1})=-\frac{5}2H_2(E_4\varphi_{-2,1})=
10E_6\varphi_{-2,1}+\frac 54 E_4\varphi_{0,1}.
\end{equation}
This equation shows that $\varphi_{0,1}$ {\it does not satisfies a 
modular differential equation of degree $2$ with respect to the heat 
operator} $H$. We will use the identity (\ref{deq-2:phi1}) in \S 3.

To obtain a differential equation of order $3$ for $\varphi_{0,1}$ one 
can continue the calculation above, but we propose a more simple method
of {\it $q^0$-cancellation}. 
Using (\ref{q0-coef}) we obtain
$$
q^0(H_4H_2H_0(\varphi_{0,1}))=
-\frac{13\cdot 9\cdot 5}{8}\zeta^{\pm 1}+\frac{21}8\cdot 10.
$$
It follows that the  $q^0$-Fourier coefficient of  the Jacobi form 
$\psi_{6,1}$ of weight $6$ in the left 
hand side of the differential equation (\ref{deq:K3}) is equal to $0$,
i.e., $\psi_{6,1}\in J_{6,1}(q)$. 
Therefore $\Delta^{-1}\psi_{6,1}\in J_{-6,1}$. The latter space is 
trivial according to the structure result (\ref{J2**}). Equation 
(\ref{deq:K3}) is proved.
\smallskip
 
Next we consider a variety $CY_5$ of dimension $5$. 
We know that $J_{2,\frac 52}=\Bbb C E_4\varphi_{-2,1}\varphi_{0,1}$. 
Using (\ref{EG:CY5}) and (\ref{q0-coef}) we obtain
$$
q^0(H_0(\varphi_{0,\frac 52}))=
-\frac{11}5(\zeta^{\pm \frac 32}-\zeta^{\pm \frac 12}),\quad
q^0(H_{-2}(\varphi_{-2,1}\varphi_{0,\frac 32}))=
-\frac{1}5(\zeta^{\pm \frac 32}+11\zeta^{\pm \frac 12}).
$$
Therefore
\begin{equation}\label{H0phi52}
H_0(\varphi_{0,\frac 52})=-\frac{11}5 E_4 \varphi_{-2,1}\varphi_{0,\frac 32},\quad
H_{-2}(\varphi_{-2,1}\varphi_{0,\frac 32})=-\frac{1}5\varphi_{0,1}\varphi_{0,\frac 32}.
\end{equation}
By Proposition \ref{prop:DH} and \ref{D-delta} we obtain
\begin{equation}
H_2H_0(\varphi_{0,\frac 52})=H_2\left(-\frac{11}5 E_4 \varphi_{-2,1}\varphi_{0,\frac 32}\right)=
\frac{44}5 E_6 \varphi_{-2,1}\varphi_{0,\frac 32}
+\frac{11}{25}E_4\varphi_{0,\frac 52}.
\end{equation}
It follows that $\varphi_{0,\frac 52}$ does not satisfy a modular 
differential equation of degree $2$ with respect to the heat operator 
$H$.  The third iteration $H_4H_2H_0(\varphi_{0,\frac 52})$ has weight 
$6$. Since $J_{6,\frac 52}=\phi_{0, \frac 32}J_{6,1}$ 
we can use the method of $q^0$-cancellation as in the case of $K3$ 
above. We have
$$
q^0(H_4H_2H_0(\varphi_{0,\frac 52}))=
-\frac{31\cdot21\cdot 11}{125}\zeta^{\pm \frac 32}+
\frac{19\cdot 9\cdot 1}{125}\cdot 11\zeta^{\pm \frac 12}.
$$
This proves  (\ref{deq:CY5}).
 
To rewrite  equations (\ref{deq:K3})
and (\ref{deq:CY5}) in terms of the heat operator we take into account
the identity $\frac{\partial}{\partial \tau}E_2=E_2^2-E_4$ and $\frac{\partial}{\partial \tau}E_4=4E_2E_4-4E_6$. The theorem is proved.
\end{proof}

\section{Modular differential equations of degree $2$} 

\subsection{The Kaneko--Zagier modular differential equation}
The Kaneko--Zagier differential equation
\begin{equation}\label{deq:KZ1}
f''(\tau)-\frac{k+1}6 E_2(\tau)f'(\tau)+\frac {k(k+1)}{12}
E'_2(\tau)f(\tau)=0
\end{equation}
appeared in \cite{KZ} (see also \cite{KK}) in connection with liftings 
of supersingular $j$-invariants of elliptic curves. 
For our purpose we use a representation of the Kaneko--Zagier equation 
in terms of the modular differential operator $D_k$. The equation (\ref{deq:KZ1}) 
is equivalent to 
\begin{equation}\label{deq:KZ2}
D_{k+2}D_k(f)-\frac {k(k+2)}{144}E_4\cdot f=0.
\end{equation}
It is evident that the Eisenstein series $E_4(\tau)$ of weight $k=4$
is a solution of the Kaneko--Zagier equation. This equation and its 
generalization for degree $3$ (see \cite{KNS})  have many applications 
in number theory and  the theory of vertex algebras.

If we change $D_k$ by the modular differential operator $H_{k}$ then  
the Eisenstein-Jacobi series $E_{4,1}(\tau, z)$
of weight $4$ and index $1$ satisfies a similar equation
$$
H_{6}H_4(E_{4,1}(\tau,z))-  \frac{77}4 E_4(\tau)E_{4,1}(\tau,z)=0,
$$
because $J_{8,1}^{hol}=\Bbb C E_4(\tau)E_{4,1}(\tau,z)$.
See \cite{K} for more solutions in index $1$ for different $k$.
We recall that we use factor $12$ before the differential operator $D$ 
in the modular correction $H_{k}$ of the heat operator in (\ref{heat}).

In this section we show that many  basic  generators of the bigraded 
ring $J_{*,*}$ of weak Jacobi forms satisfy differential 
equations of Kaneko--Zagier type with the heat operator $H$ instead of 
the derivation  $D=\frac 1{2\pi i }\frac{d\ }{d\tau}$.

\subsection{The Kaneko--Zagier type equation for $\varphi_{-2,1}$ or $\vartheta^2(\tau, z)$, and $\vartheta(\tau, z)\vartheta(\tau, 2z)$.}
In this subsection we consider two  examples of solutions  of Kaneko--Zagier type equations for Jacobi modular forms.

The Jacobi theta-series 
$
\vartheta(\tau,z)\in J_{\frac 12, \frac 12}^{hol}(v_{\eta^3}\cdot v_H)
$  
lies in the kernel of the heat operator:
$$
\left( 4\pi i \dfrac{\partial\ }{\partial \tau}-
\dfrac{\partial^2 \  }{\partial z^2}\right)\vartheta(\tau, z)=0.
$$
In the proof of Theorem \ref{deq:dim1} we obtained 
a formula for $H_0H_{-2}(\varphi_{-2,1})$  in (\ref{Hphi}). By (\ref{Hphi}),
\begin{equation}\label{eq:phi-21}
H_0H_{-2}(\varphi_{-2,1})-\frac 54 E_4\varphi_{-2,1}=0.
\end{equation}
Using Proposition \ref{prop:DH} we can rewrite the last equation   
in terms of 
$\vartheta^2$ (or in terms of  the first Jacobi cusp form of index one 
$\varphi_{10,1}=\eta^{18}(\tau)\vartheta^2(\tau,z)$ as in \cite{K})
\begin{equation}\label{eq:phi101}
H_3H_{1}(\vartheta^2)-\frac 54 E_4\vartheta^2=0\quad 
{\rm or}\quad  
H_3H_{1}(\varphi_{10,1})-\frac 54 E_4\varphi_{10,1}=0.
\end{equation}
Let us compare the last equation with the equation for the elliptic  genus of a $K3$ surface. (\ref{deq:K3}) is equivalent to  the equation of degree $3$ for the second Jacobi cusp form 
$\varphi_{12,1}=\Delta\varphi_{0,1}$ of index $1$
\begin{equation*}\label{deq:phi121}
H_{16}H_{14}H_{12}(\varphi_{12,1})-\frac{101}4E_4H_{12}(\varphi_{12,1})+10E_6\varphi_{12,1}=0.
\end{equation*}

Our second example is related to the product of two generators
(see (\ref{phi32})) of the bigraded ring of weak Jacobi forms 
$$
\varphi_{-2,1}(\tau, z)\cdot\varphi_{0,\frac 32}(\tau, z)=
\frac{\vartheta(\tau,z)\vartheta(\tau,2z)}{\eta^6(\tau)}\in 
J_{-2,\frac 52}.
$$
According to the structure results (\ref{eq:structure1})--(\ref{eq:structure2}) 
$$
J_{2,\frac 52}=\varphi_{0,\frac 32}J_{2,1}=\Bbb C E_4\varphi_{-2,1}\varphi_{0,\frac 32}.
$$
We have calculated $H_{-2}(\varphi_{-2,1}\varphi_{0,\frac 32})$
and  $H_0(\varphi_{0,\frac 52})$ in (\ref{H0phi52}). It gives us 
the equation for  the product $\varphi_{-2,1}\varphi_{0,\frac 32}$
which can be rewritten in terms of Jacobi theta-series
\begin{equation}
H_3H_1(\vartheta(\tau,z)\vartheta(\tau,2z))
-\frac{11}{25}E_4(\tau)\cdot(\vartheta(\tau,z)\vartheta(\tau,2z))=0.
\end{equation}

\subsection{Equations for $\vartheta^3(\tau, z)$ and 
$\vartheta^2(\tau, z)\vartheta(\tau, 2z)$.} It is interesting that the cube of the Jacobi theta-series satisfies an equation of Kaneko--Zagier type. To obtain it  we consider the product of two weak Jacobi forms 
$\varphi_{-1,\frac 12}\varphi_{-2,1}
={\vartheta^3}/{\eta^9}\in J_{-3,\frac 32}$.
We see that
$$
J_{1, \frac 32}=\varphi_{-1,\frac 12}J_{2,1}
=\Bbb C E_4\varphi_{-1,\frac 12}  \varphi_{-2,1}.
$$
Then we get the following identities
$$
H_{-3}(\varphi_{-1,\frac 12}\varphi_{-2,1})=
-\varphi_{-1,\frac 12}\varphi_{0,1}, \quad
H_{-1}(\varphi_{-1,\frac 12}\varphi_{0,1})=
-3E_4\varphi_{-1,\frac 12}\varphi_{-2,1}.
$$
Therefore we obtain an equation for the product 
$\varphi_{-1,\frac 12}\varphi_{-2,1}$ which we can rewrite  as 
\begin{equation}
H_{\frac 52}H_{\frac 32}(\vartheta^3(\tau,z))
-3E_4(\tau)\vartheta^3(\tau,z)=0.
\end{equation}

The next example is given by  the product 
$$\varphi_{-1,2}\varphi_{-2,1}
={\vartheta(\tau,2z)\vartheta^2(\tau,z)}/{\eta^9(\tau)}.$$
We have
$$
H_{-3}(\varphi_{-1,2}\varphi_{-2,1})=
-\frac 12\varphi_{-1,2}\varphi_{0,1},\quad
H_{-1}(\varphi_{-1,2}\varphi_{0,1})=
-\frac 52 E_4\varphi_{-1,2}\varphi_{-2,1}.
$$
Therefore 
\begin{equation}
H_{\frac 72}H_{\frac 32}(\vartheta(\tau,2z)\vartheta^2(\tau,z))-
\frac 54 E_4(\tau)\cdot (\vartheta(\tau,2z)\vartheta^2(\tau,z))=0.
\end{equation}

\subsection{Equation of degree $3$ for  $\vartheta^4(\tau, z)$.} 
We consider the fourth power of the Jacobi theta-series.  
To find a modular 
differential equation of this holomorphic Jacobi form we shall work 
with the square of the basic weak Jacobi form $ \varphi_{-2,1}$. The method 
of $q^0$-cancellation reduces calculations with Jacobi forms to 
simple manipulations with reciprocal polynomials.

\begin{theorem}
The fourth power $\vartheta^4(\tau, z)$ of the Jacobi theta-series
satisfies a modular differential equation of degree $3$
\begin{equation}\label{theta4}
H_6H_{4}H_{1}(\vartheta^4)- 
\frac{23}{4} E_4 H_{2}(\vartheta^4) 
+\frac{81}{4} E_6 \vartheta^4=0.
\end{equation}
\end{theorem}
\begin{proof}
We analyze the Jacobi form 
$\varphi_{-2,1}^2=\vartheta^4/\eta^{12}$.
We note that $J_{-2n,n}=\Bbb C \varphi_{-2,1}^n$
and $J_{-2n+2,n}=\Bbb C \varphi_{-2,1}^{n-1}\varphi_{0,1}$.
Using (\ref{heat}) we get  that for any $n\ge 2$
$$
H_{-2n}(\varphi_{-2,1}^n)
=(-n+\frac{1}2)\varphi_{-2,1}^{n-1}\varphi_{0,1}.
$$
For the third iteration of the modular differential operator we have 
$$
\psi_{2,2}=H_0H_{-2}H_{-4}(\varphi_{-2,1}^2)\in J_{2,2}
=\langle E_4\varphi_{-2,1}^2,\, 
\varphi_{-2,1}\varphi_{0,1}\rangle_\mathbb{C}.
$$
It follows that 
$H_0H_{-2}H_{-4}(\varphi_{-2,1}^2)$ is a linear combination of 
$E_4H_{-4}(\varphi_{-2,1}^2)$ and $E_6\varphi_{-2,1}^2$.
To compute the constants in the equation we find two coefficients of 
the  $q^0$-term of $\psi_{2,2}$ using (\ref{q0-coef})
$$
q^0(\psi_{2,2})=-\frac{3\cdot 7\cdot 11}8\zeta^2+\ldots+
\frac{9\cdot 5\cdot 1}{8}+\dots.
$$
As a result we obtain
$$
H_0H_{-2}H_{-4}(\phi_{-2,1}^2)=
\frac{23}{4} E_4 H_{-4}(\phi_{-2,1}^2) -\frac{81}{4} E_6\phi_{-2,1}^2.
$$
To finish the proof we apply Proposition \ref{prop:DH}.
\end{proof}

\noindent
{\bf Remark.} We note that $\vartheta^8(\tau, z)$ can be considered as a special Jacobi-Eisenstein series. In particular, there is a formula for its Fourier coefficients (see \cite{GW18}). We can show that 
$\vartheta^4(\tau,z)$ is a Jacobi-Eisenstein series with a nontrivial character of the full Jacobi group.

\section{$CY_4$, $CY_6$ and generating functions of Kac-Moody type}

\subsection{Modular differential equation for generic Jacobi forms of 
index $2$ and $3$}
We know that
$$
\begin{aligned}
\chi(CY_4;\tau, z)&\in J_{0,2}^{\Bbb Z}
=\,\langle E_4\varphi_{-2,1}^2,\, \varphi_{0,1}^2 \rangle_\Bbb Z\,
=\, \langle E_4\varphi_{-2,1}^2,\,\varphi_{0,2}\rangle_\Bbb Z,\\
\chi(CY_6;\tau, z)&\in J_{0,3}^{\Bbb Z}
=\, \langle E_6\varphi_{-2,1}^3,\,E_4\varphi_{-2,1}^2\varphi_{0,1},\,
\varphi_{0,1}^3\rangle_\Bbb Z.
\end{aligned}
$$

We analyze the algorithm of $q^0$-cancellation for weak Jacobi 
forms of index $2$.
There are four parameters in an equation of order $5$ in $H$
\begin{multline}\label{deg5}
H_8H_6H_4H_2H_0\phi_{0,2}+
aE_4H_4H_2H_0\phi_{0,2}\\+bE_6H_2H_0\phi_{0,2}
+cE_8H_0\phi_{0,2}+dE_{10}\phi_{0,2}=0
\end{multline}
for an arbitrary Jacobi form $\phi_{0,2}\in J_{0,2}$ of index $2$. 
The $q^0$-part of the Jacobi form  
$H_8H_6H_4H_2H_0(\phi_{0,2})\in J_{10,2}$
has three coefficients
$s_2\zeta^{\pm 2}+s_1\zeta^{\pm 1}+s_0$.
Using three parameters in (\ref{deg5}) one can obtain a Jacobi form 
$\xi_{10,2}\in J_{10,2}(q)$ such that  $q^0(\xi_{10,2})=0$.
Therefore 
$\Delta^{-1}\xi_{10,2}\in J_{-2,2}=\Bbb C\varphi_{-2,1}\varphi_{0,1}$.
Using the fourth parameter in the differential equation one can
annihilate the last term $\xi_{10,2}$. We have proved the following 
result.

\begin{theorem}\label{Thm:CY4}
An arbitrary Jacobi form $\phi_{0,2}$ of weight $0$ and index $2$ 
(resp. $\phi_{0,3}$ of index $3$) satisfies a linear (in the heat 
operator $H$) modular differential equation of order $5$ (resp. $7$).
\end{theorem}
\begin{proof}We have to prove the theorem for index $3$.
Using the same arguments as in the case of index $2$ we can construct a 
differential equation of degree $7$ for any weak Jacobi form 
$\phi_{0,3}\in J_{0,3}$. The $q^0$-Fourier coefficient  of the Jacobi 
form $H_{12}H_{10}\ldots H_0(\phi_{0,3})\in J_{14,3}$ 
has four coefficients.
A differential equation of degree $7$ for any $\phi_{0,3}$
contains  the dominant term indicated above and $6$ additional summands 
defined by the corresponding iterations of $H$. The coefficient at
$H_0\phi_{0,3}$ is a modular form of weight $12$. Therefore, there are 
$7$ free coefficients in the equation at modular factors  $E_4$, $E_6$, 
$E_8$, $E_{10}$, $E_{12}$, $\Delta$ and $E_{14}$. Using four 
coefficients one can get a modular form 
$\xi_{14,3}\in J_{14,3}(q)$ without $q^0$-part.
Thus,
$$
\Delta^{-1}\xi_{14,3}\in J_{2,3}=
\langle E_8\varphi_{-2,1}^3,\, E_6\varphi_{-2.1}^2\varphi_{0,1},\,
E_4\varphi_{-2.1}\varphi_{0,1}^2 \rangle_\Bbb C.
$$
With three additional parameters in the equation one can annihilate 
$\xi_{14,3}$. In this way we obtain a modular differential equation of 
degree $7$.
\end{proof}

\subsection{Reflective weak Jacobi forms of index $2$, $3$ and $4$.}
We will consider high-order modular differential equations in 
another paper. Now we analyze special differential equations  for 
the generators of the graded ring $J_{0,*}^{\Bbb Z}$ 
of weak Jacobi forms of weight $0$ and integral index.
It is the target ring for the elliptic genus of Calabi--Yau varieties 
of even dimensions according Theorem \ref{Th:EG}. 
The ring has four generators $\varphi_{0,1}$, $\varphi_{0,2}$,
$\varphi_{0,3}$, $\varphi_{0,4}$ according to Theorem \ref{thm:J0*}.
These Jacobi forms have many applications.
They are reflective, i.e., they determine Siegel 
paramodular forms with the simplest divisors which are generating 
functions of the basic Lorentzian Kac-Moody algebras of hyperbolic rank 
$3$ (see \cite{GN98}). In particular, these Jacobi forms are generating 
functions for  (super)-multiplicities  of the positive roots of the 
corresponding  Lorentzian Kac-Moody algebras.
Moreover, the $-1$-power of the automorphic Borcherds product 
defined by 
$2\varphi_{0,1}$ is the second quantized elliptic genus of K3 surface 
(see \cite{DMVV}).
The differential equation for  $\varphi_{0,1}$ was considered in 
Theorem \ref{deq:dim1}. 
It turns out that $\varphi_{0,2}$ and $\varphi_{0,4}$ also satisfy 
modular differential equations of degree $3$. 
The generator $\varphi_{0,3}$
gives us the first example of weak Jacobi forms of weight $0$ whose 
differential equation has degree $4$.
\smallskip 

The first step in the construction of a modular differential equation 
in Theorem \ref{Thm:CY4} is the annihilation of the $q^0$-part of a Jacobi  form. There exist three Jacobi forms 
in $J_{0,2}^\Bbb Z$ 
with only two parameters in the $q^0$-Fourier coefficient 
$$
\begin{aligned}
\varphi_{0,2}(\tau,z)&=\zeta^{\pm 1}+4+q(\ldots)\quad ({\rm see}\ 
(\ref{phi2})),\\
\psi_{0,2}(\tau,z)&=\zeta^{\pm 2}+22+q(\ldots)
=\varphi_{0,1}^2(\tau,z)-20\varphi_{0,2}(\tau,z),\\
\rho_{0,2}(\tau,z)&=
2\psi_{0,2}(\tau,z)-11\varphi_{0,2}(\tau,z)=2\zeta^{\pm 2}-11\zeta^{\pm 1}+q(\ldots).
\end{aligned}
$$
For them we can expect a special modular differential equation of order 
less than $5$. The Jacobi form $\psi_{0,2}(\tau,z)$ is also reflective 
and determines a Lorentzian Kac-Moody algebra. We show that each of 
them satisfies a modular differential equation of order $3$. 

First, consider the generator $\varphi_{0,2}$. 
We have $H_0(\varphi_{0,2})=-\zeta^{\pm 1}+2+q(\ldots)$.
Then, continuing to apply the modular differential operators, we get
$$
H_2H_0(\varphi_{0,2})=3\zeta^{\pm 1}-3+q(\ldots), \quad 
H_4H_2H_0(\varphi_{0,2})=-15\zeta^{\pm 1}+\frac{21}{2}+q(\ldots).
$$
From these formulas we obtain  
$$
\xi_{6,2}=4H_4H_2H_0(\varphi_{0,2})- 47E_4H_0(\varphi_{0,2})+13E_6\varphi_{0,2}=q(\ldots)\in J_{6,2}(q).
$$
It follows that $\Delta^{-1}\xi_{6,2}\in J_{-6,2}=\{0\}$
according to the structure of the algebra of Jacobi forms (see (\ref{J2**})). Therefore  $\varphi_{0,2}$ satisfies the equation
\begin{equation}\label{deq:phi02}
H_4H_2H_0(\varphi_{0,2})-\frac{47}4 E_4H_0(\varphi_{0,2})+ \frac{13}4E_6\varphi_{0,2}=0.
\end{equation}
In a similar way we obtain
$$
\begin{aligned}
H_0(\psi_{0,2})=-\frac{11}{2}\zeta^{\pm 2}+11+q(\ldots),\ 
H_2H_0(\psi_{0,2})&
=\frac{165}{4}\zeta^{\pm 2}-\frac{33}{2}+q(\ldots),
\\
H_4H_2H_0(\psi_{0,2})=-\frac{3135}{8}\zeta^{\pm 2}
+\frac{231}{8}+q(\ldots).
\end{aligned}
$$
Therefore
\begin{equation}\label{deq:psi02}
H_4H_2H_0(\psi_{0,2})-\frac{263}4E_4H_0(\psi_{0,2})
+\frac{121}4E_6\psi_{0,2}=0.
\end{equation}
For the third Jacobi form $\rho_{0,2}$ we have 
$$
H_0(\rho_{0,2})=-11\zeta^{\pm 2}+11\zeta^{\pm 1}+q(\ldots),\quad
H_2H_0(\rho_{0,2})=\frac{165}{2}\zeta^{\pm 2}-33\zeta^{\pm 1}+q(\ldots)
$$
 and 
$H_4H_2H_0(\rho_{0,2})=-\frac{3135}{4}\zeta^{\pm 2}+165\zeta^{\pm 1}+q(\ldots)$. 
Therefore
\begin{equation}\label{deq:rho02}
H_4H_2H_0(\rho_{0,2})-\frac{335}4 E_4H_0(\rho_{0,2})-\frac{275}4E_6\rho_{0,2}=0.
\end{equation}

For the last two generators $\varphi_{0,3}$ and $\varphi_{0,4}$ the 
situation is more complicated. 
We have to analyze not only the $q^0$-term, but a part
of $q^1$-Fourier coefficients. In particular, there are no modular 
differential equations of order $3$ for $\varphi_{0,3}$. 

\begin{theorem}\label{thm:deq-phi03}
The generators $\varphi_{0,3}$, $\varphi_{0,4}$ 
(see (\ref{phi2})) satisfy the following modular 
differential equations of order $4$ and $3$:
$$
H_6H_4H_2H_0(\varphi_{0,3})-\frac{29}{2}E_4H_2H_0(\varphi_{0,3})
+22E_6H_0(\varphi_{0,3})-\frac{119}{16}E_8\varphi_{0,3}=0
$$
and 
\begin{equation}\label{deq:phi04}
H_4H_2H_0(\varphi_{0,4})-\frac{107}{16}E_4H_0(\varphi_{0,4})
+\frac{23}{32}E_6\varphi_{0,4}=0.
\end{equation}
\end{theorem}
\begin{proof} 
In Theorem \ref{Thm:CY4} we proved that a general weak Jacobi 
form in $J_{0,3}$ satisfies a modular differential equation of order 
$7$. The generator $\varphi_{0,3}=\zeta^{\pm 1}+2+q(\ldots)$ 
has the simplest $q^0$-term. 
To annihilate it one needs two parameters in an equation of 
order $3$. See the proof of the equations 
(\ref{deq:phi02})--(\ref{deq:rho02}). In the case of index $3$
the subspace of the weak Jacobi forms of weight $6$ and index $3$ 
without $q^0$-Fourier coefficients  
is one-dimensional: $J_{6,3}(q)=\mathbb{C}\Delta\varphi_{-2,1}^3$.
 This is the first case when we need to control 
the $q^1$-term of Fourier expansions. We see that 
$$
\begin{aligned}
H_0(\varphi_{0,3})&=-\frac 12\zeta^{\pm 1}\ +\ 1\ 
+\ q\,\left(-7\zeta^{\pm 3}+\ldots\right),\\
H_2H_0(\varphi_{0,3})&=\ \ \frac 54\zeta^{\pm 1}\ -\ \frac 32\ 
+\ q\,\left(-\frac{21}2\zeta^{\pm 3}+\ldots\right),\\
H_4H_2H_0(\varphi_{0,3})&=-\frac {45}8\zeta^{\pm 1}+\frac {21}4\,
+\ q\,\left(\ \frac{21}4\zeta^{\pm 3}+\ldots\right).
\end{aligned}
$$
Therefore we obtain the identity
$$
H_4H_2H_0(\varphi_{0,3})-\frac{33}4 E_4H_0(\varphi_{0,3})
+\frac{3}2 E_6\varphi_{0,3}=60\,\Delta\varphi_{-2,1}^3.
$$
In this case we have  to analyze  equations of order $4$ with the dominant term $H_6H_4H_2H_0(\varphi_{0,3})\in J_{8,3}$.
We note that the subspace  
$J_{8,3}(q)=\Delta J_{-4,3}=\Bbb C \Delta\varphi_{-2,1}^2\varphi_{0,1}$ 
is again one-dimensional. Therefore, it is enough to control only one 
coefficient in $q^1$-terms. We use three formulas above and  the fourth  identity
$$
H_6H_4H_2H_0(\varphi_{0,3})=\frac {585}{16}\zeta^{\pm 1}
-\frac {231}8
+q\,\left(-\frac{105}8\zeta^{\pm 3}+\ldots\right)
$$
in order to find the equation of the theorem for $\varphi_{0,3}$.

Consider the last generator 
$\varphi_{0,4}=\frac{\vartheta(\tau, 3z)}{\vartheta(\tau, z)}$ of the graded ring $J_{0,*}^{\Bbb Z}$. The direct computation shows that the $q^0$-term of the left hand side of (\ref{deq:phi04}) is equal to zero. The subspace
$J_{6,4}(q)=\Delta J_{-6,4}=\mathbb{C} \Delta\varphi_{-2,1}^3\varphi_{0,1}$ is again one-dimensional.  But in this case of $\varphi_{0,4}$ we are lucky. Direct computation shows that the  $q^1$-part of the Fourier expansion in (\ref{deq:phi04}) vanishes. 
Therefore  $\varphi_{0,4}$ satisfies the equation
(\ref{deq:phi04}) of degree $3$.
\end{proof}
\bigskip

\noindent
\textbf{Acknowledgements.}
The first author was supported by Ministry of Science and Higher Education of the Russian Federation, agreement  075--15--2022--287 and by the M\"obius Contest Foundation for Young Scientists.
The second author was supported by the HSE University Basic Research Program and by the PRCI SMAGP (ANR-20-CE40-0026-01).

\bibliographystyle{amsplain}

\end{document}